\documentclass[11pt,reqno]{amsart}

\usepackage{indentfirst}
\usepackage{amssymb}
\usepackage{amsmath}
\usepackage{graphics}
\usepackage{epsfig}
\usepackage{MnSymbol,wasysym}

\usepackage[utf8]{inputenc}
\usepackage{latexsym}
\usepackage{graphicx}
\usepackage{cancel}

\usepackage[normalem]{ulem}

\usepackage{hyperref}

\usepackage[T1]{fontenc}
\usepackage{kantlipsum}
\usepackage[usenames,dvipsnames]{xcolor}
\usepackage[breakable, theorems, skins]{tcolorbox}
\tcbset{enhanced}

\catcode`@=11 \@addtoreset{equation}{section}
\renewcommand\theequation{\thesection.\@arabic\c@equation}
\catcode`@=12

\newcommand{\RR}{\mathbb{R}}

\newtheorem{theorem}{Theorem}[section]
\newtheorem{lemma}[theorem]{Lemma}
\newtheorem{proposition}[theorem]{Proposition}
\newtheorem{corollary}[theorem]{Corollary}
\theoremstyle{remark}

\theoremstyle{definition}

\newtheorem{definition}{Definition}[section]

\numberwithin{equation}{section}
\usepackage{hyperref}

\newcommand{\R}{\mathbf{R}}

\newcommand{\Ll}{{\mathcal L}}

\usepackage{mathtools}

\mathtoolsset{showonlyrefs=true}

\begin{document}

\title[Hardy-Sobolev]{Dynamics of dissipative solutions to the Hardy-Sobolev parabolic equation}

\author[M. Ikeda]{Masahiro Ikeda}
\address[M. Ikeda]{Graduate School of Information Science and Technology, The University of Osaka, 1–5 Yamadaoka, Suita, Osaka 565–0871, Japan}
\email{ikeda@ist.osaka-u.ac.jp}

\author[C. J. Niche]{C\'esar J. Niche}
\address[C.J. Niche]{Departamento de Matem\'atica Aplicada, Instituto de Matem\'atica. Universidade Federal do Rio de Janeiro, CEP 21941-909, Rio de Janeiro - RJ, Brazil}
\email{cniche@im.ufrj.br}

\author[G. Planas]{Gabriela Planas}
\address[G. Planas]{Departamento de Matem\'atica, Instituto de Matem\'atica, Estat\'{\i}stica e
Computa\c{c}\~ao Cient\'{\i}fica, Universidade Estadual de Campinas, Rua Sergio Buarque de Holanda, 651, 13083-859, Campinas - SP, Brazil}
\email{gplanas@unicamp.br}

\thanks{M. Ikeda has been supported by JSPS KAKENHI Grant Number JP 23K03174.  C.J. Niche acknowledges support from  PROEX - CAPES. G. Planas was partially supported by CNPq-Brazil grant 310274/2021-4, and FAPESP-Brazil grant 19/02512-5}


\subjclass[2020]{}

\date{\today}

\begin{abstract}
 We study the long-time behaviour of solutions to the Hardy-Sobolev parabolic equation in critical function spaces for any spatial dimension $ d \geq 5$. By employing the Fourier splitting method, we establish precise decay rates for dissipative solutions, meaning those whose critical norm vanishes as time approaches infinity. Our findings offer a deeper understanding of the asymptotic properties and dissipation mechanisms governing this equation.
\end{abstract}

\maketitle

\section{Introduction}

In this article, we study the long-time behaviour of solutions to the  Hardy-Sobolev parabolic equation on $\mathbb R^d$,
\begin{align}
\label{eq:HS}
\begin{cases}
\partial _t u    =  \Delta u + |x|^{-\gamma}|u|^{2^*(\gamma) - 2} u,&(x,t)\in \mathbb{R}^d\times (0,\infty),\\
u(x,0) = u_0 (x) \in \dot{H}^1 (\mathbb R^d),
\end{cases}
\end{align}
where $d \geq 3$, $ 0 < \gamma < 2$ and  $2^*(\gamma) =\frac{2(d-\gamma)}{d-2}$ is the critical Hardy-Sobolev exponent.
The energy space $\dot{H}^1 (\mathbb R^d)$ is defined by
\[\dot{H}^1 (\mathbb R^d) = \Bigl\{ f \in L^{q_c} (\mathbb{R}^d) : \| \nabla f \| _{L^2(\mathbb{R}^d)} < \infty \Bigr\}\]
where $ q_c = \frac{2d}{d-2} $ is the critical Sobolev exponent,   for which we have the embedding $ \dot{H}^1 (\mathbb R^d) \subset L^{q_c} (\mathbb R^d)
$. Equation \eqref{eq:HS} is energy-critical because the energy
\begin{equation}
\label{eq:energy}
E_\gamma( u(t)) = \frac{1}{2} \int _{\RR^d}  |\nabla u (t) |^2  \, dx - \frac{1}{2^*(\gamma)} \int _{\RR^d}  \frac{|u (t)|^{2^*(\gamma)}}{|x|^\gamma} \, dx
\end{equation}
is invariant under the natural scaling
\begin{equation}
\label{eq:scaling}
u _{\lambda} (x,t) = \lambda ^{\frac{2-\gamma}{2^*(\gamma) - 2 }} u (\lambda x, \lambda^2 t) = \lambda ^{\frac{d-2}{2}} u (\lambda x, \lambda^2 t), \quad \lambda > 0.
\end{equation}
Notice that both terms  in \eqref{eq:energy} are invariant under \eqref{eq:scaling}.

Equation \eqref{eq:HS} has attracted attention due to the rich dynamical behaviour arising from its singular nonlinear term. Its stationary version

\begin{equation}
\label{eqn:stationary}
- \Delta U = |x|^{-\gamma}|U|^{2^*(\gamma) - 2} U, \qquad x\in \mathbb{R}^d,
\end{equation}
was studied, amongst others, by H\'enon \cite{henon1973numerical}, who proposed this equation as a model for rotating stellar systems and by Ghoussoub and Moradifam \cite{MR3052352}, who studied it through variational methods. It has been known since the seminal work of Kenig and Merle \cite{MR2257393,MR2461508} for the nonlinear Schr\"odinger and wave equations, that ground states, i.e. stationary solutions in energy-critical spaces, lead to a dynamical dichotomy: those solutions whose initial datum is, in some sense, ``below'' this ground state blowup in finite or infinite time, while those with initial datum ``above'' the ground state are global and scatter. This was proved for \eqref{eq:HS} by Chikami, Ikeda and Taniguchi \cite{MR4331259}, who showed that for initial data with energy \eqref{eq:energy} less or equal than that of the ground state and negative Nehari functional \eqref{eqn:nehari-functional}, the critical $\dot{H}^1$ norm of the solution blows up in finite or grows up in infinite time (see \eqref{eqn:growup}), while if the Nehari functional is positive, the solution is dissipative, this is, global in time and such that the critical $\dot{H}^1$ norm tends to zero, see Theorem \ref{thm:global-dynamics}.

Understanding how quickly solutions decay as time approaches infinity is essential for capturing how systems stabilise, how rapidly perturbations vanish, and whether the solutions efficiently reach equilibrium. This understanding provides a link between transient dynamics and the system's long-term behaviour.
Our main goal in this article is to provide decay rates for the $\dot{H}^1$ norm of these dissipative solutions. To achieve this, we use the decay character, introduced by Bjorland and Schonbek \cite{MR2493562}. Roughly speaking, for $u_0 \in L^2 (\mathbb{R}^n)$, its decay character is $r^{\ast} = r^{\ast} (u_0)$ such that $|\widehat{u_0} (\xi)| \approx |\xi|^{r^{\ast}}$, for $|\xi| \approx 0$. This quantity characterizes the decay of dissipative linear equations with data $u_0$, see Section \ref{DecayCharacter} for details.

 We now state the main result in this article.

\begin{theorem}
\label{thm:decay} Let $d \geq 5$, $u$ be a dissipative solution of \eqref{eq:HS} and $q^{\ast} = r^{\ast} \left( \Lambda u_0\right) > -\frac{d}{2}$,  where $\Lambda = (- \Delta) ^{1/2}$.  Then  we have
\[
\Vert u (t) \Vert ^2 _{\dot{H} ^1} \leq \left\{ \begin{array}{ll}C (1+t) ^{- \min \left\{ \frac{d}{2} + q^{\ast}, 1 \right\}} , & d \leq 10 - 4 \gamma,\\
C [\ln (e+t)]^{-2}, \, & d > 10 - 4 \gamma,
 \end{array}\right.
\]
for large enough $t$.
\end{theorem}
 The estimates in Theorem \ref{thm:decay} extend results already available for the energy-critical nonlinear heat equation, i.e. equation \eqref{eq:HS} with  $\gamma = 0$, obtained by Kosloff, Niche and Planas \cite{MR4743818} for $d = 4$ and Ikeda, Niche and Planas \cite{ikeda2025globaldynamicsenergycriticalnonlinear} for $d \geq 3$. Due to the fact that we use the Rellich inequality to show that the critical norm is a Lyapunov function, we are restricted to $d \geq 5$. The singularity at $x = 0$ in the nonlinear term forces us to make some significant modifications in the proof, when compared to the case $\gamma = 0$. These are implemented through delicate estimates in Lorentz spaces. For results on the dynamics of this equation, see Chikami, Ikeda and Taniguchi \cite{CHIKAMI2022112931}, Chikami, Ikeda, Taniguchi and Tayachi \cite{MR4803464}, Hisa and Sierzega \cite{Hisa2024},  Hisa and Takahashi \cite{HISA2021822}, Ishiwata, Ruf, Sani and Terraneo \cite{10.57262/ade030-0304-141}, and references therein.

 This article is organized as follows. In Section \ref{section-technical-results} we state results and provide definitions we need for showing, in Section \ref{section-lyapunov-function},  that the critical energy is a Lyapunov function, and for proving, in Section \ref{section-proof}, our main result Theorem \ref{thm:decay}. Finally, in Appendix \ref{appendix-technical} we give the complete statement of the result concerning the existence of solutions to \eqref{eq:HS} from Chikami, Ikeda and Taniguchi \cite{MR4331259}.

\section{Technical results}

\label{section-technical-results}

\subsection{Dichotomy}
\label{dichotomy}
 We consider the integral form of problem \eqref{eq:HS}
\begin{equation}\label{integral-equation}
	u(x,t) =e^{t\Delta} u_0(x) + \int_{0}^t e^{(t-\tau)\Delta} \Bigl\{ |x|^{-\gamma} |u(x,\tau)|^{2^*(\gamma)-2} u(x,\tau)\Bigr\}\,d\tau,
\end{equation}
where
$\{e^{t\Delta}\}_{t>0}$ is the free heat semigroup, defined by
\[
	e^{t\Delta} f(x) = (G( \cdot,t) * f)(x) = \int_{\RR^d} G(x-y,t)f(y)\, dy,\quad x\in \RR^d,\ t>0,
\]
and $G:\mathbb{R}^d \times (0,\infty)\rightarrow (0,\infty)$ is the heat kernel, i.e.,
\[	G(x,t) = (4 \pi t)^{-\frac{d}{2}} e^{-\frac{|x|^2}{4t}},\quad   x\in \mathbb R^d,\ t>0 .
\]

A function $ u = u(x,t) $  is defined as a mild solution to \eqref{eq:HS} on $  \mathbb{R}^d \times [0, T)$ with initial data $ u_0 \in \dot{H}^1(\mathbb{R}^d) $ if $ u \in C([0, T^\prime]; \dot{H}^1(\mathbb{R}^d)) $ satisfies the integral equation \eqref{integral-equation} for any $ T^\prime \in (0, T) $, where $ T \in (0, \infty]$. If $ T < \infty $, the solution \( u \) is called local in time. The maximal existence time of the solution with initial data $ u_0 $ is denoted by $ T_m = T_m(u_0)$.
The solution $ u $ is called global in time if $ T_m = +\infty $, and it is said to blow up in finite time if $ T_m < +\infty $. Furthermore, $ u $ is said to be dissipative if $T_m = +\infty $ and
\[
\lim_{t \to \infty} \|u(t)\|_{\dot{H}^1} = 0.
\]
On the other hand, $ u $ is said to grow up at infinite time if $T_m = +\infty$ and
\begin{equation}
\label{eqn:growup}
\limsup_{t \to \infty} \|u(t)\|_{\dot{H}^1} = +\infty.
\end{equation}

This problem is well-posed in the Hadamard sense, as demonstrated by Chikami, Ikeda, and Taniguchi \cite{MR4331259}. More precisely, for any initial data 
$ u_0 \in \dot{H}^1(\mathbb{R}^d)$ there exists a maximal existence time $ T_m = T_m(u_0)$ 
 such that \eqref{eq:HS} has a unique mild solution on 
$[0,T_m)$. Furthermore, the solution depends continuously on the initial data. Under certain conditions, such as when 
$\|u_0\|_{\dot{H}^1} $
  is sufficiently small, the solution extends globally in time. For a complete statement of this result, refer to the Appendix \ref{appendix-technical}.

The Nehari functional $J_\lambda:\dot{H}^1(\mathbb{R}^d)\rightarrow \mathbb{R}$ is given by 
\begin{equation}
\label{eqn:nehari-functional}
	J_\lambda(f) = \frac{d}{d\lambda} E \left( \lambda f \right) \Big|_{\lambda=1}
	= \|f\|_{\dot{H}^1}^2
	- \int_{\mathbb{R}^d} \frac{|f(x)|^{2^*(\gamma)}}{|x|^\gamma}dx, \ \ \ f\in \dot{H}^1(\mathbb{R}^d).
\end{equation}
and the Nehari manifold is defined by
\[
\mathcal{N}_\gamma:=\left\{\phi \in\dot{H}^1\left(\mathbb{R}^d\right) \backslash\{0\} ; J_\gamma(\phi)=0\right\}.
\]
 Then, the mountain pass energy $l_{\mathrm{HS}}$ is given by
\[
l_{\mathrm{HS}}:=\inf _{\phi \in \dot{H}^1\left(\mathbb{R}^d\right) \backslash\{0\}} \max _{\lambda \geq 0} E(\lambda \phi)=\inf _{\phi \in \mathcal{N}_\gamma} E_\gamma(\phi) .
\]
 The function
\begin{equation}\label{eqn:ground-state}
W_\gamma(x):=((d-\gamma)(d-2))^{\frac{d-2}{2(2-\gamma)}}\left(1+|x|^{2-\gamma}\right)^{-\frac{d-2}{2-\gamma}}
\end{equation}
is a ground state of \eqref{eq:HS}, i.e. a solution to the stationary problem \eqref{eqn:stationary}. By invariance of \eqref{eq:HS}, its scaling and rotation
\[
\mathrm{e}^{\mathrm{i} \theta_0} \lambda_0^{\frac{d-2}{2}} W\left(\lambda_0 x\right), \quad \lambda_0>0, \quad \theta_0 \in \mathbb{R}
\]
is also a ground state of \eqref{eq:HS}.
The $l_{\mathrm{HS}}$ coincides with the energy $E_\gamma\left(W_\gamma\right)$ of the ground state.

A necessary and sufficient condition on initial data at or below the ground state that dichotomizes the behavior of solutions was established by Chikami, Ikeda, and Taniguchi \cite[Thm 1.1]{MR4331259}.

\begin{theorem} \label{thm:global-dynamics}
 Let $d \geq 3,\ 0<\gamma<2$, and $u=u(t)$ be a solution to \eqref{eq:HS} with initial data $u_0 \in \dot{H}^1(\mathbb{R}^d)$. Assume $E_\gamma\left(u_0\right) \leqslant l_{\mathrm{HS}}$. Then, the following statements hold:
 \begin{itemize}
\item[(a)] If $J_\gamma\left(u_0\right)>0$, then $u$ is dissipative.
\item[(b)] If $J_\gamma\left(u_0\right)<0$, then $u$ blows up in finite time or grows up at infinite time. Furthermore, if $u_0 \in L^2(\mathbb{R}^d)$ is also satisfied, then $u$ blows up in finite time.
    \end{itemize}
\end{theorem}

\subsection{Hardy-Sobolev and Rellich inequalities}

The Hardy-Sobolev inequality plays a crucial role throughout this paper.

\begin{lemma}[Thm 15.1.1 and Thm 15.2.2,  Ghoussub and Moradifam \cite{MR3052352}]  \label{lem:HS-ineq}  Let $d \geq 3$ and $ 0 \leqslant \gamma \leq 2$. Then, the inequality
\[
\left(\int_{\mathbb{R}^d} \frac{|f(x)|^{2^*(\gamma)}}{|x|^\gamma} \mathrm{d} x\right)^{\frac{1}{2 *(\gamma)}} \leq C_{\mathrm{HS}}\left(\int_{\mathbb{R}^d}|\nabla f(x)|^2 \mathrm{~d} x\right)^{\frac{1}{2}}
\]
holds for any $f \in \dot{H}^1(\mathbb{R}^d)$, where $C_{\mathrm{HS}}=C_{\mathrm{HS}}(d, \gamma)$ is the best constant which is attained  by the extremal $W_\gamma$ given in \eqref{eqn:ground-state}.
\end{lemma}

We will also use the following Rellich inequality.

\begin{lemma}[Rellich \cite{MR88624,MR240668}] \label{lem:Rellich-ineq} 
 Assume $d \geq 5$. There exists a constant $C>0$ such that
\[
\int_{\mathbb{R}^d} \frac{|f(x)|^{2}}{|x|^4} d x \leq \frac{16}{d^2(d-4)^2}\int_{\mathbb{R}^d}|\Delta f(x)|^2 d x, \quad \text { for all } f \in H^2(\mathbb{R}^d).
\]
\end{lemma}

\subsection{Lorentz spaces}

We define the distribution function $d_f$ of a function $f$ by
\[
d_f(\lambda):=|\{x \in \mathbb{R}^d ;|f(x)|>\lambda\}|,
\]
where $|A|$ denotes the Lebesgue measure of a set $A$ and by $f^*$  the decreasing rearrangement of $f$ given by
\[
f^*(t):=\inf \left\{\lambda>0 ; d_f(\lambda) \leq t\right\}.
\]
For a $ f $  measurable function  define,  for $0<q, r \leq \infty$,
\[
\|f\|_{L^{q, r}}:= \begin{cases}\displaystyle\left(\int_0^{\infty}\left(t^{\frac{1}{q}} f^*(t)\right)^r \frac{d t}{t}\right)^{\frac{1}{r}} & \text { if } r<\infty \\ \displaystyle\sup _{t>0} t^{\frac{1}{q}} f^*(t) & \text { if } r=\infty.\end{cases}
\]
 The set of all $f$  with $ \|f\|_{L^{q, r}(\mathbb{R}^d)}<\infty $ denoted by $L^{q, r} (\mathbb{R}^d)$
 is called the Lorentz space with indices $  q$ and $r$,  see Grafakos \cite{MR2445437} and Lemarie-Rieusset \cite{MR1938147}  for properties of such spaces.

We now state a generalized H\"older inequality. 
\begin{lemma}
	\label{lem:Holder}
Let $0 < q, q_1,q_2 < \infty$ and $0<r,r_1,r_2 \le \infty$.
Then the following assertions hold:
\begin{itemize}
\item[(i)] If
\[
\frac{1}{q} = \frac{1}{q_1} + \frac{1}{q_2}\quad \text{and}\quad \frac{1}{r} \le \frac{1}{r_1} + \frac{1}{r_2},
\]
then
there exists a constant $C>0$ such that
\[
\|f g\|_{L^{q,r}}
\le C \|f\|_{L^{q_1,r_1}}\|g\|_{L^{q_2,r_2}}
\]
for any $f \in L^{q_1,r_1}(\mathbb R^d)$ and $g \in L^{q_2,r_2}(\mathbb R^d)$.

\item[(ii)]
There exists a constant $C>0$ such that
\[
\|f g\|_{L^{q,r}}
\le C \|f\|_{L^{q,r}}\|g\|_{L^{\infty}}
\]
for any $f \in L^{q,r}(\mathbb R^d)$ and $g \in L^{\infty}(\mathbb R^d)$.
\end{itemize}
\end{lemma}

We will use the following Sobolev critical embedding (see Appendix in  Alvino, Trombetti and Lions \cite{MR979040}).
\begin{lemma} \label{critical-embedding-lorenz}  Let $ d \geq 3$.  Then, for any $ f \in \dot{H}^1(\mathbb{R}^d)$, there exists  $C > 0$ such that
\begin{equation*} 
\| f \|_{L^{q_c,2}}\leq C \| f \|_{\dot{H}^1},
\end{equation*}
where $ q_c = \frac{2d}{d-2} $ is the critical Sobolev exponent.
\end{lemma}

We state some  estimates for the heat semigroup.

\begin{proposition}[Prop 3.1, Chikami, Ikeda, Taniguchi and              Tayachi \cite{MR4803464}]	\label{prop:linear-main}
Let $d \in \mathbb N$, $1\le q_1\le \infty$, $1<q_2\le \infty$, $0< r_1,r_2 \le \infty$ and $s_1,s_2\in \mathbb R$.
Then there exists a constant $C>0$ such that
\begin{equation*} 
\||x|^{s_2}e^{t\Delta}f\|_{L^{q_2,r_2}}\le C t ^{-\frac{d}{2} (\frac{1}{q_1} - \frac{1}{q_2}) - \frac{s_1 - s_2}{2}}\||x|^{s_1}f\|_{L^{q_1,r_1}}
\end{equation*}
for any $t>0$ if $q_1,q_2,r_1,r_2,s_1,s_2$ satisfy
\begin{align*}
  & 0 \le \frac{s_2}{d} + \frac{1}{q_2} \le \frac{s_1}{d} + \frac{1}{q_1} \le 1, \label{linear-condi1}\\
  & s_2 \le s_1,
\end{align*}
and
\begin{align*}
  & r_1 \le 1\quad \text{if }\frac{s_1}{d} + \frac{1}{q_1}=1 \text{ or }q_1=1,
  \\
  & r_2 = \infty\quad \text{if } \frac{s_2}{d} + \frac{1}{q_2}=0,
  \\
  & r_1 \le r_2 \quad \text{if }\frac{s_1}{d} + \frac{1}{q_1}=\frac{s_2}{d} + \frac{1}{q_2},
  \\
  & r_i=\infty \quad \text{if } q_i=\infty \quad (i=1,2).
\end{align*}
\end{proposition}

From this proposition with $s_1=\gamma$, $s_2=0$ and $f=|x|^{-\gamma}g$, we have the following.
\begin{corollary}
	\label{cor:linear-main}
Let $d \in \mathbb N$, $1\le q_1\le \infty$, $1<q_2\le \infty$, $0< r_1,r_2 \le \infty$ and $\gamma\ge 0$.
Then there exists a constant $C>0$ such that
\begin{equation*}
\|e^{t\Delta}(|x|^{-\gamma}g)\|_{L^{q_2,r_2}}\le C t ^{-\frac{d}{2} (\frac{1}{q_1} - \frac{1}{q_2}) - \frac{\gamma}{2}}\|g\|_{L^{q_1,r_1}}
\end{equation*}
for any $t>0$ if $q_1,q_2,r_1,r_2$ satisfy
\begin{equation*}
  0 \le \frac{1}{q_2} \le \frac{\gamma}{d} + \frac{1}{q_1} \le 1, 
\end{equation*}
and
\begin{align*}
  & r_1 \le 1\quad \text{if }\frac{\gamma}{d} + \frac{1}{q_1}=1 \text{ or }q_1=1, 
  \\
  & r_2 = \infty\quad \text{if } q_2=\infty,
  \\
  & r_1 \le r_2 \quad \text{if }\frac{\gamma}{d} + \frac{1}{q_1}=\frac{1}{q_2}, 
  \\
  & r_i=\infty \quad \text{if } q_i=\infty \quad (i=1,2).
\end{align*}
\end{corollary}

\subsection{Gronwall inequalities}

We will need the following Gronwall-type inequalities.

\begin{proposition}[Theorem 1, page 356, Mitrinovi\'{c}, Pe\v{c}ari\'{c} and Fink \cite{MR1190927}] \label{gronwall-1} Let $x,k:J \to \RR$ continuous and $a,b: J \to \R$ Riemann integrable in $J = [\alpha, \beta]$. Suppose that $b, k \geq 0$ in $J$. Then, if
\begin{displaymath}
x(t) \leq a(t) + b(t) \int _{\alpha} ^t k(s) x(s) \, ds, \quad t \in J
\end{displaymath}
then
\begin{displaymath}
x(t) \leq a(t) + b(t) \int _{\alpha} ^t a(s) k(s) \exp \left( \int_s ^t b(r)k(r)  \, dr \right) \, ds, \quad t \in J.
\end{displaymath}
\end{proposition}

\begin{proposition}[Corollary 1.2,  page 4,  Bainov and Simeonov \cite{MR1171448}] \label{gronwall-2} Let $a,b,\psi,:J \to \RR$ continuous in $J = [\alpha, \beta]$ and $b \geq 0$.  If $a(t)$ is nondecreasing then
\begin{displaymath}
\psi(t) \leq a(t) +  \int _{\alpha} ^t b(s) \psi(s) \, ds, \quad t \in J
\end{displaymath}
implies
\begin{displaymath}
\psi(t) \leq a(t)  \exp \left( \int_{\alpha} ^t b(s)  \, ds \right) \quad t \in J.
\end{displaymath}
\end{proposition}

 \subsection{Decay Character}  \label{DecayCharacter} The decay character,  introduced by Bjorland and M.E. Schonbek \cite{MR2493562} and studied further by Niche and M.E. Schonbek \cite{MR3355116},  and Brandolese \cite{MR3493117},  associates to $v_0 \in L^2(\RR^d)$ a number $-\frac{d}{2} < r^{\ast} = r^{\ast} (v_0) < \infty$ that characterizes upper and lower bounds of the $L^2$ norm of solutions to the heat equation with such initial data.  Roughly speaking,  $r^{\ast} = r$  if $|\widehat{v_0} (\xi)| \approx |\xi|^{r}$ at $\xi = 0$.   We now recall the definition and properties of the decay character.

\begin{definition} \label{decay-indicator}
Let  $v_0 \in L^2(\RR^d)$. For $r \in \left(- \frac{d}{2}, \infty \right)$, we define the {\em decay indicator}  $P_r (v_0)$ corresponding to $v_0$ as
\begin{displaymath}
P_r(v_0) = \lim _{\rho \to 0} \rho ^{-2r-d} \int _{B(\rho)} \bigl |\widehat{v_0} (\xi) \bigr|^2 \, d \xi,
\end{displaymath}
provided this limit exists and    $B(\rho)$ refers to  the ball at the origin with radius $\rho$.
\end{definition}

\begin{definition} \label{df-decay-character} The {\em decay character of $ v_0$}, denoted by $r^{\ast} = r^{\ast}( v_0)$ is the unique  $r \in \left( -\frac{d}{2}, \infty \right)$ such that $0 < P_r (v_0) < \infty$, provided that this number exists. We set $r^{\ast} = - \frac{d}{2}$, when $P_r (v_0)  = \infty$ for all $r \in \left( - \frac{d}{2}, \infty \right)$  or $r^{\ast} = \infty$, if $P_r (v_0)  = 0$ for all $r \in \left( -\frac{d}{2}, \infty \right)$.
\end{definition}

The decay character can be computed in some important cases,  such as $v_0 \in L^p (\mathbb R^d) \cap L^2 (\mathbb R^d)$,  where $1 \leq p < 2$,  see M.E. Schonbek \cite{MR837929} and Ferreira, Niche and Planas \cite{MR3565380}.

Next Theorem describes the decay of solutions to the heat equation in terms of the decay character.

\begin{theorem} [Theorem 5.8, Bjorland and M.E.Schonbek \cite{MR2493562}]
\label{characterization-decay-l2}
Let $v_0 \in L^2 (\RR^d)$ has decay character $r^{\ast} (v_0) = r^{\ast}$.  Let $v (t)$ be a solution to heat equation with initial datum $v_0$.  Then if $- \frac{d}{2 } < r^{\ast}< \infty$,  there exist constants $C_1, C_2> 0$ such that
\begin{displaymath}
C_1 (1 + t)^{- \left( \frac{d}{2} + r^{\ast} \right)} \leq \Vert v(t) \Vert _{L^2} ^2 \leq C_2 (1 + t)^{- \left( \frac{d}{2} + r^{\ast} \right)}.
\end{displaymath}
\end{theorem}

A decay characterization analogous to that in Theorem \ref{characterization-decay-l2} can be obtained for more general diagonalizable linear operators $\Ll$,  i.e.  those such that its symbol matrix $\widehat{M}$  in $\widehat{\Ll u} = \widehat{M} \widehat{u}$ is such that $\widehat{M} = - O^t \cdot |\xi|^{2 \alpha} Id _{\RR^d} \cdot O$,  with $O \in O(d)$,  see Niche and M.E. Schonbek \cite{MR3355116}. The decay character can also provide estimates for linear systems that are not diagonalizable,  for example the damped wave equation,  see C\'ardenas and Niche \cite{CARDENAS2022125548},  or the  Navier-Stokes-Voigt equations,  see Niche \cite{MR3437593}.

In Definition \ref{decay-indicator} we assume the existence of a certain limit leading to a positive $P_r (u_0)$ and then to a decay character in Definition \ref{df-decay-character}.  That limit may not exist for some $v_0 \in L^2 (\RR ^d)$,  as  Brandolese \cite{MR3493117} showed  by constructing initial data with fast oscillations near the origin in frequency space for which Definition \ref{decay-indicator} does not hold.  To circumvent this problem,  he  introduced the idea of an upper and lower decay character that,  when equal,  recover the  one  in Definition \ref{df-decay-character}.  He also proved that the decay character exists if and only if $v_0$ belongs to an explicit subset $\dot{\mathcal{A}} ^{- \left( \frac{d}{2} + r^{\ast} \right)} _{2, \infty}$ in the critical homogeneous Besov space  $\dot{B} ^{- \left( \frac{d}{2} + r^{\ast} \right)} _{2, \infty}$,  and that this happens if and only if solutions to linear diagonalizable systems as above have sharp algebraic decay.

\section{\texorpdfstring{\ensuremath{\dot{H}^{1}}}{H1} norm is a Lyapunov function}

\label{section-lyapunov-function}

The following result is key in the proof of the decay rates in  Theorem \ref{thm:decay}.

\begin{proposition}
\label{lyapunov}
Let $  d \geq 5$. The critical norm $\Vert \cdot \Vert _{\dot{H} ^1}$ is a Lyapunov function for dissipative solutions to \eqref{eq:HS}, for large enough $t$.
\end{proposition}

\begin{proof} We begin by noting that from the proof of Proposition 2.9 in Chikami, Ikeda and Taniguchi  \cite{MR4331259}, we obtain that for $d\geq 4$ or $ d=3 $ and $ 0 \leq \gamma < 3/2$, solutions to \eqref{eq:HS} are such that for any $t_0 \in (0,\infty)$
\begin{equation*} 
\partial_t u,\  \Delta u , \ |x|^{-\gamma}|u|^{2^*(\gamma) - 2} u \in L^2 _{loc} ( [t_0, \infty),  L^2 (\RR^d) ).
\end{equation*} This together with Lemma 5.10 in Bahouri,  Chemin and Danchin \cite{MR2768550} allow us to  deduce the energy equality
\begin{displaymath}
\Vert u(t_2) \Vert ^2 _{\dot{H} ^1}  + 2 \int _{t_1} ^{t_2} \Vert \nabla u(\tau) \Vert ^2  _{\dot{H}^1} \, d\tau = \Vert u (t_1) \Vert ^2 _{\dot{H} ^1} + 2  \int _{t_1} ^{t_2} \langle  u(\tau) , |x|^{-\gamma}|u(\tau)|^{2^*(\gamma)-2}u(\tau) \rangle _{\dot{H}^1} \, d \tau,
\end{displaymath}
 for any $t_1 < t_2$ in a compact interval.
 
 By the Parseval identity,   Lemmas \ref{lem:Holder}  and  \ref{critical-embedding-lorenz}  we have
\begin{align*}
\langle \Lambda u,\Lambda(|x|^{-\gamma}|u|^{2^*(\gamma)-2}u)\rangle_{L^2}
& =\sum_{j=1}^d\langle \partial_{x_j}u,\partial_{x_j}(|x|^{-\gamma}|u|^{2^*(\gamma)-2}u)\rangle_{L^2}\\
&\le C\|\Lambda u\|_{L^{\frac{2d}{d-2},2}}\sum_{j=1}^d\|\partial_{x_j}(|x|^{-\gamma}|u|^{2^*(\gamma)-2}u)\|_{L^{\frac{2d}{d+2},2}}\\
&\le C\|\Lambda u\|_{\dot{H}^1}\sum_{j=1}^d\|\partial_{x_j}(|x|^{-\gamma}|u|^{2^*(\gamma)-2}u)\|_{L^{\frac{2d}{d+2},2}}.
\end{align*}
Since a direct computation gives
\begin{align*}
\partial_{x_j}(|x|^{-\gamma}|u|^{2^*(\gamma)-2}u)
=-\gamma|x|^{-\gamma-2}x_j|u|^{2^*(\gamma)-2}u
+(2^*(\gamma)-2)|x|^{-\gamma}|u|^{2^*(\gamma)-3}u\partial_{x_j}u+|x|^{-\gamma}|u|^{2^*(\gamma)-2}\partial_{x_j}u,
\end{align*}
the following estimate holds
\begin{align*}
|\partial_{x_j}(|x|^{-\gamma}|u|^{2^*(\gamma)-2}u)|
\le C|x|^{-\gamma-1}|u|^{2^*(\gamma)-1}+C|x|^{-\gamma}|u|^{2^*(\gamma)-2}|\partial_{x_j}u|.
\end{align*}
For the first term, by using again  Lemmas \ref{lem:Holder} and \ref{critical-embedding-lorenz},
\begin{align*}
    \||x|^{-\gamma-1}|u|^{2^{*}(\gamma)-2}u\|_{L^{\frac{2d}{d+2},2}}
    &\le \||x|^{-\gamma}\|_{L^{\frac{d}{\gamma},\infty}}\||u|^{2^*(\gamma)-2}\|_{L^{\frac{2d}{4-2\gamma},\infty}}\||x|^{-1}u\|_{L^{\frac{2d}{d-2},2}}\\
    &\le C\|u\|_{L^{\frac{2d}{d-2},\infty}}^{2^{*}(\gamma)-2}\|\nabla(|x|^{-1}u)\|_{L^2}\\
    &\le C\|u\|_{\dot{H}^1}^{2^{*}(\gamma)-2}(\||x|^{-2}u\|_{L^2}+\||x|^{-1}\nabla u\|_{L^2})\\
    & \le C\|u\|_{\dot{H}^1}^{2^*(\gamma)-2}\|\Lambda u\|_{\dot{H}^1}
\end{align*}
where we have used Lemmas \ref{lem:Rellich-ineq} and  \ref{lem:HS-ineq}.  For the second term,  we apply again  Lemmas \ref{lem:Holder}   and   \ref{critical-embedding-lorenz}, to arrive at

\begin{align*}
\||x|^{-\gamma}|u|^{2^*(\gamma)-2}\partial_{x_j}u\|_{L^{\frac{2d}{d+2},2}}
&\le \||x|^{-\gamma}\|_{L^{\frac{d}{\gamma},\infty}}\||u|^{2^*(\gamma)-2}\|_{L^{\frac{2d}{4-2\gamma},\infty}}\|\partial_{x_j}u\|_{L^{\frac{2d}{d-2},2}}\\
&\le C\|u\|_{L^{\frac{2d}{d-2},\infty}}^{2^{*}(\gamma)-2}\|\Lambda u\|_{\dot{H}^1}\\
& \le  C\|u\|_{\dot{H}^1}^{2^*(\gamma)-2}\|\Lambda u\|_{\dot{H}^1}.
\end{align*}
Consequently,
\begin{equation}\label{eq:estimate-nonlinear}
  \langle \Lambda u,\Lambda(|x|^{-\gamma}|u|^{2^*(\gamma)-2}u)\rangle_{L^2}\le C\|u\|_{\dot{H}^1}^{2^*(\gamma)-2}\|\nabla u\|_{\dot{H}^1}^2.
\end{equation}
Plugging this estimate in the energy equality there follows
\begin{displaymath}
\Vert u(t_2) \Vert ^2 _{\dot{H} ^1}  + 2 \int _{t_1} ^{t_2} \Vert \nabla u(\tau) \Vert ^2  _{\dot{H}^1} \, d\tau \leq \Vert u (t_1) \Vert ^2 _{\dot{H} ^1} + 2  \int _{t_1} ^{t_2} C\|u(\tau)\|_{\dot{H}^1}^{2^*(\gamma)-2}\|\nabla u(\tau)\|_{\dot{H}^1}^2 \, d \tau.
\end{displaymath}
As the solution is dissipative,  we can take $ T >0$ large enough such that $\| u(t) \|_{\dot{H}^1}$ is small for all $t \geq T$.  Hence, we infer that, for any $ T < t_1 < t_2$,
\[
\Vert u(t_2) \Vert ^2 _{\dot{H} ^1} \leq \Vert u(t_1) \Vert ^2 _{\dot{H} ^1} \,,
\]
concluding that the $\dot{H}^1$-norm is a Lyapunov function for large values of $t$.
\end{proof}

 \section{Proof of Theorem \ref{thm:decay}}

\label{section-proof}
 
We split the proof into several steps.

\subsection{A differential inequality}
We know, by Proposition \ref{lyapunov}, that   the $\dot{H} ^1$ norm is a nonincreasing function, hence has a derivative a.e. Then
\begin{align*}
\notag \frac{d}{dt} \Vert u(t) \Vert^2 _{\dot{H} ^1} & = 2 \langle \Lambda  u(t), \partial_t \Lambda u(t) \rangle  = 2 \left \langle \Lambda  u (t) ,\Lambda  \left( \Delta u (t)  +  |x|^{-\gamma}|u(t)|^{2^*(\gamma)-2}u(t) \right) \right  \rangle \notag \\ & \leq - 2   \left(1 -  C \Vert  u(t) \Vert^{2^*(\gamma)-2} _{\dot{H} ^1} \right) \Vert \nabla u(t) \Vert^2 _{\dot{H} ^1}  \leq - \widetilde{C} \Vert \nabla u(t) \Vert^2 _{\dot{H} ^1},
\end{align*}
where we have used \eqref{eq:estimate-nonlinear} and that the solution is dissipative to have a small enough $\| u (t)\|_{\dot{H}^1}$ for sufficiently large $t$.

Having this inequality, we now apply the Fourier Splitting Method, introduced by M.E. Schonbek, to analyze the energy decay in solutions to parabolic conservation laws \cite{MR571048} and the Navier-Stokes equations \cite{MR775190}, \cite{MR837929}. This method relies on the observation that, for many dissipative equations, the remaining energy at sufficiently large times is primarily concentrated in the low-frequency region. 
Let
$ B(t) $ be a ball centered at the origin in frequency space, with a continuously varying, time-dependent radius 
$r(t)$, such that
\begin{displaymath}
B(t) = \left\{\xi \in \mathbb{R}^d: |\xi| \leq r(t) = \left( \frac{g'(t)}{\widetilde{C} g(t)} \right) ^{\frac{1}{2}}  \right\},
\end{displaymath}
where $g$  is a continuous, increasing function satisfying 
 $g(0) = 1$.   By Plancherel's Theorem and using that \begin{displaymath}
- \widetilde{C} \Vert \nabla u(t) \Vert^2 _{\dot{H} ^1} \leq  - \frac{g'(t)}{g(t)} \int _{B(t) ^c} ||\xi| \widehat{u} (\xi,t) |^2 \, d \xi
\end{displaymath}
we deduce 
\begin{align*}
 \frac{d}{dt} \Vert u(t) \Vert^2 _{\dot{H} ^1} & \leq - \widetilde{C} \Vert \nabla u(t) \Vert^2 _{\dot{H} ^1} =  - \widetilde{C} \int _{\mathbb{R} ^d }|\xi|^2 \,  ||\xi| \widehat{u} (\xi,t)|^2 \, d \xi  \\  & \leq     \frac{g'(t)}{g(t)}  \int _{B(t)}  ||\xi| \widehat{u} (\xi,t)|^2 \, d \xi -  \frac{g'(t)}{g(t)}  \int _{\mathbb{R} ^d}  ||\xi|  \widehat{u} (\xi,t)|^2 \, d \xi.
\end{align*}
Next, we multiply both sides by 
$g(t)$
and rearrange to derive 
\begin{equation}
\label{eqn:key-inequality}
\frac{d}{dt}   \left( g(t)  \Vert u (t) \Vert _{\dot{H} ^1} ^2 \right) \leq g'(t)  \int _{B(t)} ||\xi|  \widehat{u} (\xi, t)| ^2 \, d \xi.
\end{equation}
To bound the right-hand side, we notice that
\begin{align}
\int _{B(t)} ||\xi|  \widehat{u} (\xi, t)| ^2 \, d \xi & \leq C \int _{B(t)} \left| e^{- t |\xi| ^2} |\xi| \widehat{u_0} (\xi) \right|^2 \, d \xi \notag \\ & + C \int _{B(t)} \left( \int _0 ^t  e^{- (t-s) |\xi| ^2} |\xi| \mathcal{F} \left[|x|^{-\gamma}|u|^{2^*(\gamma)-2}u \right] (\xi, s) \, ds \right) ^2 \, d \xi. \notag
\end{align}
The first term corresponds to the linear part, namely the heat equation, and can therefore be estimated using Theorem \ref{characterization-decay-l2} as
\begin{equation}
\label{eq:decay-linear-part}
\int _{B(t)} \left| e^{- t |\xi| ^2} |\xi| \widehat{u_0} (\xi) \right|^2 \, d \xi \leq C \Vert e^{t \Delta} \Lambda u_0 \Vert^2 _{L^2} \leq C (1 + t) ^{- \left( \frac{d}{2} + q^{\ast} \right)},
\end{equation}
where $q^{\ast} = r^{\ast} \left( \Lambda u_0 \right)$.

We next  deal with the nonlinear term.  By Fubini's Theorem
\begin{align*}
& \int_{B(t)}\left(\int_0^te^{-(t-s)|\xi|^2}|\xi|\mathcal{F}[|x|^{-\gamma}|u|^{2^*(\gamma)-2}u]ds\right)^2d\xi \\
&\le r(t)^2\int_{B(t)}\left(\int_{0}^t|\mathcal{F}[e^{(t-s)\Delta}(|x|^{-\gamma}|u|^{2^*(\gamma)-2}u)]|ds\right)^2d\xi\\
&=r(t)^2\int_{0}^t\int_{0}^t\int_{B(t)}|\mathcal{F}[e^{(t-s)\Delta}(|x|^{-\gamma}|u|^{2^*(\gamma)-2}u)]||\mathcal{F}[e^{(t-s')\Delta}(|x|^{-\gamma}|u|^{2^*(\gamma)-2}u)]| \, d\xi \, ds \, ds'\\
&\le r(t)^2\int_{0}^t\int_{0}^t\left(\int_{B(t)}|\mathcal{F}[e^{(t-s)\Delta}(|x|^{-\gamma}|u|^{2^*(\gamma)-2}u)]|^2d\xi\right)^{1/2}  \times \\
& \ \ \ \ \ \ \ \left(\int_{B(t)}|\mathcal{F}[e^{(t-s')\Delta}(|x|^{-\gamma}|u|^{2^*(\gamma)-2}u)]|^2d\xi\right)^{1/2} \, ds \, ds'.
\end{align*}
To estimate the term
\begin{align*}
    \int_{B(t)}|\mathcal{F}[e^{(t-s)\Delta}|x|^{-\gamma}|u(s)|^{2^*(\gamma)-2}u(s)]|^2d\xi
\end{align*}
we will use Lemma \ref{lem:Holder}. Let $\alpha\ge d/2$ and $\alpha'$ such that $1/\alpha+1/\alpha'=1$ and then pick $r\in (\alpha,\infty]$ and $r'$ satisfying $1/r+1/r' = 1$. So, we have
\begin{align*}
    \int_{B(t)}  |\mathcal{F}[e^{(t-s)\Delta}(|x|^{-\gamma}|u|^{2^*(\gamma)-2}u)]|^2d\xi
     & \le \|1\|_{L^{\alpha,r}(B(t))}\||\mathcal{F}[e^{(t-s)\Delta}(|x|^{-\gamma}|u|^{2^*(\gamma)-2})]|^2\|_{L^{\alpha',r'}(B(t))}\\
    &\le C r(t)^{d/\alpha}\|\mathcal{F}[e^{(t-s)\Delta}(|x|^{-\gamma}|u|^{2^*(\gamma)-2}u)]\|_{L^{2\alpha',2r'}}^2 
    \\ & \le C
    r(t)^{d/\alpha}\|e^{(t-s)\Delta}(|x|^{-\gamma}|u|^{2^*(\gamma)-2}u)]\|_{L^{\frac{2\alpha}{\alpha+1},2r'}}^2.
\end{align*}
Now, we employ the estimate of the heat semigroup in Corollary \ref{cor:linear-main}.
Let $\beta:=\frac{2d}{d+2-2\gamma}$ and $\mu\ge \frac{2(d-2)}{d+2-2\gamma}$. Since $\alpha\ge d/2$
\begin{align*}
\|e^{(t-s)\Delta}(|x|^{-\gamma}|u|^{2^*(\gamma)-2}u)\|_{L^{\frac{2\alpha}{\alpha+1},2r'}}
&\le C (t-s)^{-\frac{d}{2}\left(\frac{1}{\beta}-\frac{\alpha+1}{2\alpha}\right)-\frac{\gamma}{2}}\|u^{2^*(\gamma)-2}u\|_{L^{\beta,\mu}}
\\ & \le C (t-s)^{-\frac{1}{2}\left(1-\frac{d}{2\alpha}\right)} \|u\|_{L^{\frac{2d}{d-2},2}}^{2^*(\gamma)-1}
 \\ & \le C (t-s)^{-\frac{1}{2}\left(1-\frac{d}{2\alpha}\right)} C\|u(s)\|_{\dot{H}^1}^{2^*(\gamma)-1} ,
 \end{align*}
where we have used  Lemma \ref{critical-embedding-lorenz}. Then,  we obtain the  estimate for the nonlinear term
\begin{align*}
\int_{B(t)}\left(\int_0^te^{-(t-s)|\xi|^2}|\xi|\mathcal{F}[|x|^{-\gamma}|u|^{2^*(\gamma)-2}u]ds\right)^2d\xi 
& \le r(t)^{2+\frac{d}{\alpha}}\left(\int_{0}^t(t-s)^{-\frac{1}{2}\left(1-\frac{d}{2\alpha}\right)}\|u(s)\|_{\dot{H}^1}^{2^*(\gamma)-1}ds\right)^2\\
&\le r(t)^{2+\frac{d}{\alpha}}t^{\frac{d}{2\alpha}}\int_0^t\|u(s)\|_{\dot{H}^1}^{2(2^*(\gamma)-1)}ds.
\end{align*}

Hence, we arrive at the differential inequality
\begin{align}
\label{eqn:main-estimate}
\frac{d}{dt}   \left( g(t)  \Vert u (t) \Vert _{\dot{H} ^1} ^2 \right)   \leq g'(t) (1 + t) ^{- \left( \frac{d}{2} + q^{\ast} \right)}   + g'(t) r(t)^{2+\frac{d}{\alpha}}t^{\frac{d}{2\alpha}}\int_0^t\|u(s)\|_{\dot{H}^1}^{2(2^*(\gamma)-1)}ds.
\end{align}

We observe that the parameter $\alpha$  will be   chosen  later in order to optimize the decay.

\subsection{Preliminary decay}
Let
\begin{displaymath}
g(t) = [\ln (e+t)]^k,
\end{displaymath}
for some large $k$ to be determined later. Then
\begin{displaymath}
r(t) = C \left(\frac{1}{ (e+t) \ln (e+t)}  \right) ^{\frac{1}{2}}.
\end{displaymath}
As $\Vert u(t) \Vert _{\dot{H} ^1} \leq C$ for all $ t >0$, using this $r(t)$  in \eqref{eqn:main-estimate} we obtain
\begin{align}
\frac{d}{dt}   \left( [\ln (e+t)]^k  \Vert u (t) \Vert _{\dot{H} ^1} ^2 \right)  & \leq C \frac{[\ln (e+t)]^{k-1}}{e+t}(1 + t) ^{- \left( \frac{d}{2} + q^{\ast} \right)}   \notag \\ & + C \frac{[\ln (e+t)]^{k-1}}{e+t} \, \left(\frac{1}{(e+t) \ln (e+t)}  \right) ^{ 1 +\frac{d}{ 2 \alpha}} t^{\frac{d}{2\alpha}+1} \notag \\
&\le  C \frac{[\ln (e+t)]^{k-1}}{e+t}(1 + t) ^{- \left( \frac{d}{2} + q^{\ast} \right)}  \, + C \frac{[\ln (e+t)]^{k-2-\frac{d}{2 \alpha}}}{e+t}.
\end{align}
Now we integrate on both sides. Since, for $ k > 1$
\[
\int _0 ^t \frac{\left[\ln (e+s) \right]^{k-1}}{e+s}  (1 + s) ^{- \left( \frac{d}{2} + q^{\ast} \right)} \, ds  \leq C \int _1 ^{\ln (e+s)} z^{k-1} \, e^{-(\frac{d}{2}+ q^{\ast})z} dz \leq C,
\]
the crucial term is
\begin{displaymath}
\int_0 ^t  \frac{[\ln (e+s)]^{k-2-\frac{d}{2 \alpha}}}{e+s} \, ds = \int_1 ^{\ln (e+t)} z ^{k - 2 - \frac{d}{2 \alpha}} \, dz \leq C [\ln (e+t)] ^ {k - 1 - \frac{d}{2 \alpha}},
\end{displaymath}
where we have chosen $k$ large enough such that  $k - 1 - \frac{d}{2 \alpha} > 0$.  Thus,
the whole bound obtained after integrating is
\begin{align}
\label{eqn:first-estimate}
 \Vert u (t) \Vert _{\dot{H} ^1} ^2  & \leq C [\ln (e+t)]^{-k}+ C [\ln (e+t)] ^ {- \left( 1 +  \frac{d}{2 \alpha} \right)}  \nonumber \\&  \leq C [\ln (e+t)] ^ {- \left( 1 + \frac{d}{2 \alpha} \right)}=C [\ln (e+t)] ^ {- 2},
\end{align}
where in the last inequality we took $2 \alpha = d$.  This is our preliminary decay,  which is valid for any $0 < \gamma < 2$ and any $d \geq 5$,  because $\gamma$ appeared only in   $\|u(s)\|_{\dot{H}^1}^{2^*(\gamma)-1}$.   We will use this estimate to bootstrap and obtain tighter decay bounds.

\subsection{Bootstrap: setting}
\label{bootstrap-section}
As
\begin{displaymath}
2 < 2 \left( 2^*(\gamma)-1 \right) < 2 \frac{d+2}{d - 2},
\end{displaymath}
we can rewrite
\begin{displaymath}
2 \left( 2^*(\gamma)-1 \right) = 2 + \left( 2 \left( 2^*(\gamma)-1 \right) - 2\right) = 2 + 4 \frac{2 - \gamma}{d - 2}.
\end{displaymath}
Then, writing
\begin{equation}
\label{eqn:log-decay-gronwall}
\|u(s)\|_{\dot{H}^1}^{2 \left( 2^*(\gamma)-1 \right)} = \|u(s)\|_{\dot{H}^1}^{2 + 2 \frac{4 - 2 \gamma}{d - 2}} = \|u(s)\|_{\dot{H}^1}^{2} \|u(s)\|_{\dot{H}^1}^{ 2 \frac{4 - 2 \gamma}{d - 2}},
\end{equation}
by using \eqref{eqn:first-estimate} we arrive at
\begin{equation} \label{eqn:log-decay-gronwall2}
\|u(s)\|_{\dot{H}^1}^{2 \left( 2^*(\gamma)-1 \right)}
\leq C \|u(s)\|_{\dot{H}^1}^2 \, [\ln (e+s)] ^ {-4 \frac{2 -  \gamma}{d - 2}}.
\end{equation}
 Now let $g(t) = C (t+1) ^m$,  for a large enough $m$ to be chosen later. Then,  using  \eqref{eqn:log-decay-gronwall2} in \eqref{eqn:main-estimate} gives us
\begin{align}
\frac{d}{dt}   \left((t+1) ^m \Vert u (t) \Vert _{\dot{H} ^1} ^2 \right)  & \leq C (t+1) ^{m-1} (t + 1) ^{- \left( \frac{d}{2} + q^{\ast} \right)}  \notag \\ & + C (t+1) ^{m-1} \, (t+1)^{- \left( 1+\frac{d}{2 \alpha} \right)} \, t^{\frac{d}{2\alpha}} \int_0 ^t \|u(s)\|_{\dot{H}^1}^2 \, [\ln (e+s)] ^ {- 4 \frac{2 -  \gamma}{d - 2}} \, ds  .
\end{align}
Assume that  $m > \max\{\frac{d}{2} + q^{\ast}, 1\}$ so, after integrating, we arrive at
\begin{align} \label{eq:bootstrap}
\Vert u (t) \Vert _{\dot{H} ^1} ^2 & \leq C (t+1) ^{-m} + C (t + 1) ^{- \left( \frac{d}{2} + q^{\ast} \right)} + C (t+1) ^{-1} \int_0 ^t \|u(s)\|_{\dot{H}^1}^2 \, [\ln (e+s)] ^ {- 4 \frac{2 -  \gamma}{d - 2}} \, ds \notag\\
& \leq C (t + 1) ^{- \left( \frac{d}{2} + q^{\ast} \right)} + C (t+1) ^{-1} \int_0 ^t \|u(s)\|_{\dot{H}^1}^2 \, [\ln (e+s)] ^ {- 4 \frac{2 -  \gamma}{d - 2}} \, ds.
\end{align}

To use an appropriate Gronwall inequality,  we consider two different cases, $4 \frac{2 -  \gamma}{d - 2} > 1$ and $4 \frac{2 -  \gamma}{d - 2} =  1$ and  estimate each case separately.    When $4 \frac{2 -  \gamma}{d - 2} < 1$ we do not obtain an improvement in decay,  thus the decay rate is given by \eqref{eqn:first-estimate}.

Recall that
\begin{equation*}
\int_s ^t (r + 1) ^{-1} \, [\ln (e+r)] ^{- \nu} \, dr  \leq
\begin{cases}
C,  & \text{ if } \nu > 1, \\
C\ln \left( \ln (e+t) \right), & \text { if } \nu = 1, \\
C [\ln (e+t)] ^{1 - \nu}, & \text { if } \nu < 1,
\end{cases}
\end{equation*}
which leads to
\begin{equation*}
\exp \left( \int_s ^t (r + 1) ^{-1} \, [\ln (e+r)] ^{- \nu} \, dr \right)  \leq
\begin{cases}
C,  & \text{ if } \nu > 1, \\
[\ln (e+t)]^C, & \text { if } \nu = 1, \\
\exp \left( C[\ln (e+t)] ^{1 - \nu} \right), & \text { if } \nu < 1.
\end{cases}
\end{equation*}

\subsection{Bootstrap: 1st case  \texorpdfstring{\ensuremath{4 \frac{2 -  \gamma}{d - 2} > 1}}{4 (2 -  gamma)/(d - 2) > 1}}
We first suppose $q^{\ast} >1-\frac{d}{2}$.   Then consider
\begin{align}
x(t) & =  \Vert u (t) \Vert _{\dot{H} ^1} ^2,  \quad a(t) = C (1 + t) ^{- \left( \frac{d}{2} + q^{\ast} \right)} \notag \\ b(t) & = C(1+t)^{-1},  \quad k(t) = \frac{1}{[\ln (e+t)]^{4 \frac{2 -  \gamma}{d - 2}}}.  \notag
\end{align}
As  $4 \frac{2 -  \gamma}{d - 2} > 1$,  we have
\begin{align}
\label{eqn:integral-kb}
\int_s ^t b(r) \, k(r) \, dr & = C\int _s ^t \frac{1}{(1 + r) [\ln (e+r)]^{4 \frac{2 -  \gamma}{d - 2}}} \, dr  \leq C,  \notag \\ \int_0 ^t a(s) \, k(s) \, ds & =C \int _0 ^t \frac{1}{(1 + s) ^{\frac{d}{2} + q^{\ast}} [\ln (e+s)]^{4 \frac{2 -  \gamma}{d - 2}}} \, ds  \leq C,
\end{align}
and thus applying Proposition \ref{gronwall-1} in \eqref{eq:bootstrap} yields a faster decay
\begin{equation}\label{firstdecay}
\Vert u (t) \Vert _{\dot{H} ^1} ^2  \leq  C (1 + t) ^{- \left( \frac{d}{2} + q^{\ast} \right)} + C (1+t)^{-1} \leq  C (1+t)^{-1}.
\end{equation}

Now,  consider $q^{\ast} \leq 1-\frac{d}{2}$.  We rewrite \eqref{eq:bootstrap} as
\begin{equation}
(t + 1) \Vert u (t) \Vert _{\dot{H} ^1} ^2 \leq C (t + 1) ^{1 - \left( \frac{d}{2} + q^{\ast} \right)} + C  \int_0 ^t \frac{(s+1) \|u(s)\|_{\dot{H}^1}^2}{(s+1) [\ln (e+s)] ^ {4 \frac{2 -  \gamma}{d - 2}}} \, ds.
\end{equation}
So,  if
\begin{displaymath}
\psi (t) = (1 + t)  \Vert u (t) \Vert _{\dot{H} ^1} ^2,  \quad a(t) = C (1 + t) ^{1 - \left( \frac{d}{2} + q^{\ast} \right)},
\end{displaymath}
\begin{displaymath} 
b(t) = \frac{ C}{(1 + t)  [\ln (e+t)]^{4 \frac{2 -  \gamma}{d - 2}}},
\end{displaymath}
the fact that $b(s)$ is such that

\begin{displaymath}
\int_r ^t b(s) \, ds \leq C,
\end{displaymath}
and $q^{\ast} \leq 1-\frac{d}{2}$ lead to
\begin{displaymath}
(1 + t)  \Vert u (t) \Vert _{\dot{H} ^1} ^2  \leq  C (1 + t) ^{1 - \left( \frac{d}{2} + q^{\ast} \right)},
\end{displaymath}
where we used Proposition \ref{gronwall-2}.

Hence, together with \eqref{firstdecay}, there follows
\begin{displaymath}
 \Vert u (t) \Vert _{\dot{H} ^1} ^2  \leq  C (1 + t) ^{-\min \{ \frac{d}{2} + q^{\ast} ,1\}}.
\end{displaymath}

\subsection{Bootstrap: 2nd case \texorpdfstring{\ensuremath{4 \frac{2 -  \gamma}{d - 2} = 1}}{4 (2 -  gamma)/(d - 2) = 1}}
Once again, we first suppose $q^{\ast} >1-\frac{d}{2}$.   Then consider
\begin{align}
x(t) & =  \Vert u (t) \Vert _{\dot{H} ^1} ^2,  \quad a(t) = C (1 + t) ^{- \left( \frac{d}{2} + q^{\ast} \right)} \notag \\ b(t) & = C(1+t)^{-1},  \quad k(t) = \frac{1}{\ln (e+t)}.  \notag
\end{align}
We now have
\begin{align}
\int_s ^t b(r) \, k(r) \, dr & = C\int _s ^t \frac{1}{(1 + r) \ln (e+r)} \, dr  \leq C \ln \ln (e+t),  \notag \\ \int_0 ^t a(s) \, k(s) \, ds & =C \int _0 ^t \frac{1}{(1 + s) ^{\frac{d}{2} + q^{\ast}} \ln (e+s)} \, ds  \leq C,
\end{align}
and thus applying Proposition \ref{gronwall-1} in \eqref{eq:bootstrap} yields 
\begin{displaymath}
\Vert u(t) \Vert _{\dot{H}^1} ^2  \leq C (1 + t) ^{- \left( \frac{d}{2} + q^{\ast} \right)}  + C (t+1) ^{-1}[\ln(e+t)]^C
 \leq C (t+1) ^{-1}[\ln(e+t)]^C.
\end{displaymath}
We will use this decay to bootstrap again.
So, we go back to \eqref{eqn:log-decay-gronwall} and use  the previous estimate to have 
\begin{equation}
\|u(s)\|_{\dot{H}^1}^{2 \left( 2^*(\gamma)-1 \right)} = \|u(s)\|_{\dot{H}^1}^{2 + 1} \leq C \|u(s)\|_{\dot{H}^1}^2 \, (1 + s) ^{-\frac{1}{2}} [\ln (e+s)]^\frac{C}{2}.
\end{equation}
We proceed as before, arriving at
\begin{displaymath}
\Vert u (t) \Vert _{\dot{H} ^1} ^2 \leq C (t + 1) ^{- \left( \frac{d}{2} + q^{\ast} \right)} + C (t+1) ^{-1} \int_0 ^t \|u(s)\|_{\dot{H}^1}^2 \,  (1 + s) ^{-\frac{1}{2}} [\ln (e+s)]^\frac{C}{2}  \, ds.
\end{displaymath}
Then consider
\begin{align}
x(t) & =  \Vert u (t) \Vert _{\dot{H} ^1} ^2,  \quad a(t) = C (1 + t) ^{- \left( \frac{d}{2} + q^{\ast} \right)} \notag \\ b(t) & = C(1+t)^{-1},  \quad k(t) = \frac{[\ln (e+t)]^\frac{C}{2}}{(1+t)^{\frac{1}{2}}}.  \notag
\end{align}
Note that as $q^{\ast}  > 1-\frac{d}{2}$
\begin{align}
\int_s ^t b(r) \, k(r) \, dr & = C\int _s ^t \frac{[\ln (e+r)]^\frac{C}{2}}{(1 + r)^{\frac{3}{2}}} \, dr  \leq C ,  \notag \\ \int_0 ^t a(s) \, k(s) \, ds & =C \int _0 ^t \frac{[\ln (e+s)]^\frac{C}{2}}{(1 + s) ^{\frac{1}{2}+\frac{d}{2} + q^{\ast}} } \, ds  \leq C.
\end{align}
Hence, Proposition \ref{gronwall-1} gives us
\begin{equation}\label{fasterdecay}
\Vert u(t) \Vert _{\dot{H}^1} ^2  \leq C (1 + t) ^{- \left( \frac{d}{2} + q^{\ast} \right)}  + C (t+1) ^{-1}  \leq C (t+1) ^{-1}.
\end{equation}

Now assume $q^{\ast}  \leq 1-\frac{d}{2}$.  As in the previous case,  rewrite \eqref{eq:bootstrap} as
\begin{equation}
(t + 1) \Vert u (t) \Vert _{\dot{H} ^1} ^2 \leq C (t + 1) ^{1 - \left( \frac{d}{2} + q^{\ast} \right)} + C  \int_0 ^t \frac{(s+1) \|u(s)\|_{\dot{H}^1}^2}{(s+1) \ln (e+s)} \, ds.
\end{equation}
So,  if
\begin{displaymath}
\psi (t) = (1 + t)  \Vert u (t) \Vert _{\dot{H} ^1} ^2,  \quad a(t) = C (1 + t) ^{1 - \left( \frac{d}{2} + q^{\ast} \right)}, 
\end{displaymath}
\begin{displaymath}
b(t) = \frac{ C}{(1 + t)  \ln (e+t)},
\end{displaymath}
the fact that $b(s)$ is such that
\begin{displaymath}
\int_r ^t b(s) \, ds \leq C \ln (\ln (e+t)),
\end{displaymath}
and $q^{\ast} \leq 1-\frac{d}{2}$ lead,  through Proposition \ref{gronwall-2} to
\begin{displaymath}
\Vert u (t) \Vert _{\dot{H} ^1} ^2  \leq  C (1 + t) ^{- \left( \frac{d}{2} + q^{\ast} \right)} [\ln (e+t)]^C.
\end{displaymath}
We will use this now to bootstrap again.
Go all the way back to \eqref{eqn:log-decay-gronwall}.   Using the estimate just obtained we get
\begin{equation}
\|u(s)\|_{\dot{H}^1}^{2 \left( 2^*(\gamma)-1 \right)} = \|u(s)\|_{\dot{H}^1}^{2 + 1} \leq C \|u(s)\|_{\dot{H}^1}^2 \, \left( (1 + t) ^{- \left( \frac{d}{2} + q^{\ast} \right)} [\ln (e+t)]^C \right) ^{\frac{1}{2}}.
\end{equation}
We proceed as before, and we obtain
\begin{displaymath}
\Vert u (t) \Vert _{\dot{H} ^1} ^2 \leq C (t + 1) ^{- \left( \frac{d}{2} + q^{\ast} \right)} + C (t+1) ^{-1} \int_0 ^t \|u(s)\|_{\dot{H}^1}^2 \, \left( (1 + s) ^{- \left( \frac{d}{2} + q^{\ast} \right)} [\ln (e+s)]^C \right) ^{\frac{1}{2}} \, ds,
\end{displaymath}
which we turn into
\begin{displaymath}
(1+t) \Vert u (t) \Vert _{\dot{H} ^1} ^2 \leq C (t + 1) ^{1 - \left( \frac{d}{2} + q^{\ast} \right)} + C \int_0 ^t \frac{(1+s) \|u(s)\|_{\dot{H}^1}^2 [\ln (e+s)]^{\frac{C}{2}}}{  (1 + s) ^{1+\frac{1}{2} \left( \frac{d}{2} + q^{\ast} \right)}  } \, ds.
\end{displaymath}
Take
\begin{displaymath}
\psi (t) = (1 + t)  \Vert u (t) \Vert _{\dot{H} ^1} ^2,  \quad a(t) = C (1 + t) ^{1 - \left( \frac{d}{2} + q^{\ast} \right)}, 
\end{displaymath}
\begin{displaymath}
b(t) = C \frac{[\ln(e+t)]^{\frac{C}{2}}}{(1 + t) ^{1 + \frac{1}{2} \left(\frac{d}{2} + q^{\ast}  \right)}}.
\end{displaymath}
Then $b(s)$ is such that

\begin{displaymath}
\int_r ^t b(s) \, ds \leq C,
\end{displaymath}
so from Proposition \ref{gronwall-2} we obtain

\begin{displaymath}
(1 + t)  \Vert u (t) \Vert _{\dot{H} ^1} ^2 \leq C (1 + t) ^{1 - \left( \frac{d}{2} + q^{\ast} \right)}.
\end{displaymath}
Plugging with \eqref{fasterdecay} furnishes 
\begin{displaymath}
 \Vert u (t) \Vert _{\dot{H} ^1} ^2  \leq  C (1 + t) ^{-\min \{ \frac{d}{2} + q^{\ast} ,1\}}.
\end{displaymath}
This finishes the proof.  $\Box$

\appendix

\section{}
\label{appendix-technical}

For the reader's convenience, we recall here 
 the existence result for mild solutions to \eqref{eq:HS} obtained by Chikami, Ikeda and Taniguchi \cite[Prop 2.5]{MR4331259}.

  Let $T \in (0,\infty]$, $q\in [1,\infty]$, and $\alpha\in \RR$. The space $\mathcal{K}^{q,\alpha}(T)$ is defined by
\[
   \mathcal{K}^{q,\alpha}(T):=\left\{u\in \mathcal{D}'((0,T)\times\mathbb R^d)\ ;\ \|u\|_{\mathcal{K}^{q,\alpha}(T')}
   <\infty\ \text{for any }T' \in (0,T)\right\}
\]
endowed with the norm
\[
\|u\|_{\mathcal K^{q,\alpha}(T)}
	=\sup_{0\le t\le T}t^{\frac{d}{2}(\frac{1}{q_c}-\frac{1}{q})+\alpha}\|u\|_{L^q},
\]
where $ \mathcal{D}'((0,T)\times\mathbb R^d)$ is the space of distributions on $[0,T)\times\mathbb R^d$.
We simply write $\mathcal{K}^{q}(T)=\mathcal{K}^{q,0}(T)$ when $\alpha=0$, and $\mathcal{K}^{q,\alpha}=\mathcal{K}^{q,\alpha}(\infty)$ and $\mathcal{K}^{q}=\mathcal{K}^{q}(\infty)$
when $T=\infty$.

\begin{proposition}[Well-posedness in the energy space]
\label{prop:wellposed1} Let $ d \geq 3 $ and $0 \leq \gamma < 2$.
Assume that $q\in (1,\infty)$ satisfies
\begin{equation} \label{l:crtHS.nonlin.est:c2}
	\frac{1}{q_c} - \frac{1}{d(2^*(\gamma)-1)}
	< \frac{1}{q} < \frac{1}{q_c}.
\end{equation}
Then, the following statements hold:
\begin{itemize}
\item[(i)] (Existence)
For any $u_0 \in \dot{H}^1(\RR^d)$, there exists a maximal existence time $T_m=T_m(u_0)\in (0,\infty]$ such that
there exists a unique mild solution
\[
	u\in C([0,T_m); \dot{H}^1(\RR^d))\cap \mathcal K^q(T_m)
\]
to \eqref{eq:HS} with $u(0)=u_0$. Moreover, the solution $u$ satisfies
\[
\|u\|_{\mathcal K^{\tilde{q}}_{\tilde{r}}(T)} =
	\left(
		\int_0^T (t^{\kappa} \|u(t)\|_{L^{\tilde{q}}})^{\tilde{r}} \, dt
	\right)^\frac{1}{\tilde{r}} < \infty
\]
for any $T \in (0,T_m)$ and for any $\tilde{q}, \tilde{r} \in [1,\infty]$ satisfying \eqref{l:crtHS.nonlin.est:c2} and
\begin{equation}\label{condi-new}
	0\le  \frac1{\tilde r}  < \frac{d}2 \left(\frac1{q_c} - \frac1{\tilde q} \right),
\end{equation}
where $\kappa$ is given by
\[
\kappa = \kappa(\tilde{q},\tilde{r}) = \frac{d}{2} \left( \frac{1}{q_c} - \frac{1}{\tilde q}\right) - \frac{1}{\tilde r}.
\]

\item[(ii)] (Uniqueness in $\mathcal{K}^{q}(T)$)
Let $T>0.$ If $u_1, u_2 \in \mathcal{K}^{q}(T)$
satisfy the integral equation \eqref{integral-equation}
with $u_1(0)=u_2(0)=u_0$, then $u_1=u_2$ on $[0,T].$

\item[(iii)] (Continuous dependence on initial data)
The map $T_m : \dot{H}^1(\RR^d) \to (0,\infty]$ is
lower semicontinuous. Furthermore, for any $u_0, v_0 \in \dot{H}^1(\RR^d)$ and for any $T < \min\{T_m(u_0),T_m(v_0)\}$, there exists a constant $C>0$, depending on $\|u_0\|_{\dot H^1}$, $\|v_0\|_{\dot H^1}$, and $T$, such that
\[
	\sup_{t \in [0,T]} \|u(t) - v(t)\|_{\dot H^1}
	+
	\|u-v\|_{\mathcal K^q(T)} \le C \|u_0 - v_0\|_{\dot H^1}.
\]

\item[(iv)] (Blow-up criterion)
If $T_m < + \infty,$ then $\|u\|_{\mathcal{K}^q(T_m)} = \infty.$

\item[(v)] (Small-data global existence and dissipation)
There exists $\rho >0$ such that
if $u_0 \in \dot{H}^1(\RR^d)$ satisfies
$\|e^{t\Delta} u_0\|_{\mathcal{K}^{q}}\le \rho$,
then $T_m=+\infty$ and
\[
\|u\|_{\mathcal{K}^{q}} \le 2\rho \quad \text{and}\quad \lim_{t\to\infty}\|u(t)\|_{\dot H^1} = 0.
\]

\item[(vi)] (Dissipation of global solutions)
The following statements are equivalent:
\begin{itemize}
	\item[(a)] $T_m=+\infty$ and $\|u\|_{\mathcal K^q} < \infty$.
	\item[(b)] $\lim_{t\to T_m} \|u(t)\|_{\dot H^1}=0$.
	\item[(c)] $\lim_{t\to T_m} t^{\frac{d}{2} (\frac{1}{q_c}-\frac{1}{q})} \|u(t)\|_{L^{q}}=0$.
\end{itemize}

\item[(vii)] Let $d=3$. Suppose that $q$ satisfies the additional assumption
\[	\frac{1}{q_c} - \frac{1}{24}
	< \frac{1}{q}.\]
Then, for any $u_0 \in \dot H^1(\RR^3)$, there exists a maximal existence time $T_m=T_m(u_0)\in (0,\infty]$ such that
there exists a unique mild solution
\[
	u\in C([0,T_m); \dot{H}^1(\RR^d))\cap \mathcal K^q(T_m)\quad \text{and}\quad \partial_t u \in \mathcal K^{3,1}(T_m)
\]
to \eqref{eq:HS} with $u(0)=u_0$. Furthermore, the solution $u$ satisfies
\[
	\partial_t u \in \mathcal K^{2,1}(T_m).
\]
\end{itemize}
\end{proposition}

The definition of $ \mathcal{K}^q$, combined with standard estimates for the heat semigroup and the Sobolev Embedding Theorem, implies that
\begin{displaymath}
 \| e^{t\Delta} u_0 \|_{\mathcal{K}^q} \leq C \Vert u_0 \Vert _{L ^{q_c}} \leq C \Vert u_0 \Vert _{\dot{H}^1},
\end{displaymath}
which, in turn, leads to the existence of a global solution for sufficiently small
$\Vert u_0 \Vert _{\dot{H}^1}$.


\begin{thebibliography}{10}

\bibitem{MR979040}
Angelo Alvino, Guido Trombetti, and Piere~L. Lions.
\newblock On optimization problems with prescribed rearrangements.
\newblock {\em Nonlinear Anal.}, 13(2):185--220, 1989.

\bibitem{MR2768550}
Hajer Bahouri, Jean-Yves Chemin, and Rapha\"{e}l Danchin.
\newblock {\em Fourier analysis and nonlinear partial differential equations},
  volume 343 of {\em Grundlehren der mathematischen Wissenschaften [Fundamental
  Principles of Mathematical Sciences]}.
\newblock Springer, Heidelberg, 2011.

\bibitem{MR1171448}
Drumi Ba\u{\i}nov and Pavel Simeonov.
\newblock {\em Integral inequalities and applications}, volume~57 of {\em
  Mathematics and its Applications (East European Series)}.
\newblock Kluwer Academic Publishers Group, Dordrecht, 1992.
\newblock Translated by R. A. M. Hoksbergen and V. Covachev [V. Khr. Kovachev].

\bibitem{MR2493562}
Clayton Bjorland and Mar\'{\i}a~E. Schonbek.
\newblock Poincar\'e's inequality and diffusive evolution equations.
\newblock {\em Adv. Differential Equations}, 14(3-4):241--260, 2009.

\bibitem{MR3493117}
Lorenzo Brandolese.
\newblock Characterization of solutions to dissipative systems with sharp
  algebraic decay.
\newblock {\em SIAM J. Math. Anal.}, 48(3):1616--1633, 2016.

\bibitem{CARDENAS2022125548}
Armando~S. C\'ardenas and César~J. Niche.
\newblock Decay character and estimates for the damped wave equation.
\newblock {\em J. Math. Anal. Appl.}, 506(1):125548, 2022.

\bibitem{MR4331259}
Noboru Chikami, Masahiro Ikeda, and Koichi Taniguchi.
\newblock Well-posedness and global dynamics for the critical {H}ardy-{S}obolev
  parabolic equation.
\newblock {\em Nonlinearity}, 34(11):8094--8142, 2021.

\bibitem{CHIKAMI2022112931}
Noboru Chikami, Masahiro Ikeda, and Koichi Taniguchi.
\newblock Optimal well-posedness and forward self-similar solution for the
  {H}ardy-{H}\'enon parabolic equation in critical weighted {L}ebesgue spaces.
\newblock {\em Nonl. Analysis}, 222:112931, 2022.

\bibitem{MR4803464}
Noboru Chikami, Masahiro Ikeda, Koichi Taniguchi, and Slim Tayachi.
\newblock Unconditional uniqueness and non-uniqueness for {H}ardy-{H}\'enon
  parabolic equations.
\newblock {\em Math. Ann.}, 390(3):3765--3825, 2024.

\bibitem{MR3565380}
Lucas C.~F. Ferreira, C\'{e}sar~J. Niche, and Gabriela Planas.
\newblock Decay of solutions to dissipative modified quasi-geostrophic
  equations.
\newblock {\em Proc. Amer. Math. Soc.}, 145(1):287--301, 2017.

\bibitem{MR3052352}
Nassif Ghoussoub and Amir Moradifam.
\newblock {\em Functional inequalities: new perspectives and new applications},
  volume 187 of {\em Mathematical Surveys and Monographs}.
\newblock American Mathematical Society, Providence, RI, 2013.

\bibitem{MR2445437}
Loukas Grafakos.
\newblock {\em Classical {F}ourier analysis}, volume 249 of {\em Graduate Texts
  in Mathematics}.
\newblock Springer, New York, second edition, 2008.

\bibitem{henon1973numerical}
Michel H{\'e}non.
\newblock Numerical experiments on the stability of spherical stellar systems.
\newblock {\em Astronomy and Astrophysics, Vol. 24, p. 229 (1973)}, 24:229,
  1973.

\bibitem{Hisa2024}
Kotaro Hisa and Mikolaj Sierzega.
\newblock Existence and nonexistence of solutions to the {H}ardy parabolic
  equation.
\newblock {\em Funkcialaj Ekvacioj}, 67(2):149--174, 2024.

\bibitem{HISA2021822}
Kotaro Hisa and Jin Takahashi.
\newblock Optimal singularities of initial data for solvability of the {H}ardy
  parabolic equation.
\newblock {\em J. Differential Equations}, 296:822--848, 2021.

\bibitem{ikeda2025globaldynamicsenergycriticalnonlinear}
Masahiro Ikeda, César~J. Niche, and Gabriela Planas.
\newblock Global dynamics for the energy-critical nonlinear heat equation.
\newblock {\em arXiv:2412.04238}, 2024.

\bibitem{10.57262/ade030-0304-141}
Michinori Ishiwata, Bernhard Ruf, Federica Sani, and Elide Terraneo.
\newblock {Blow-up and global solutions for subcritical and critical parabolic
  equations in ${\mathbb R}^N$}.
\newblock {\em Adv. Differential Equations}, 30(3/4):141 -- 176, 2025.

\bibitem{MR2257393}
Carlos~E. Kenig and Frank Merle.
\newblock Global well-posedness, scattering and blow-up for the
  energy-critical, focusing, non-linear {S}chr\"{o}dinger equation in the
  radial case.
\newblock {\em Invent. Math.}, 166(3):645--675, 2006.

\bibitem{MR2461508}
Carlos~E. Kenig and Frank Merle.
\newblock Global well-posedness, scattering and blow-up for the energy-critical
  focusing non-linear wave equation.
\newblock {\em Acta Math.}, 201(2):147--212, 2008.

\bibitem{MR4743818}
Leonardo Kosloff, C\'{e}sar~J. Niche, and Gabriela Planas.
\newblock Decay rates for the 4{D} energy-critical nonlinear heat equation.
\newblock {\em Bull. Lond. Math. Soc.}, 56(4):1468--1482, 2024.

\bibitem{MR1938147}
Pierre~G. Lemari\'{e}-Rieusset.
\newblock {\em Recent developments in the {N}avier-{S}tokes problem}, volume
  431 of {\em Chapman \& Hall/CRC Research Notes in Mathematics}.
\newblock Chapman \& Hall/CRC, Boca Raton, FL, 2002.

\bibitem{MR1190927}
Dragoslav~S. Mitrinovi\'{c}, Josip~E. Pe\v{c}ari\'{c}, and A.~M. Fink.
\newblock {\em Inequalities involving functions and their integrals and
  derivatives}, volume~53 of {\em Mathematics and its Applications (East
  European Series)}.
\newblock Kluwer Academic Publishers Group, Dordrecht, 1991.

\bibitem{MR3437593}
C\'esar~J. Niche.
\newblock Decay characterization of solutions to {N}avier-{S}tokes-{V}oigt
  equations in terms of the initial datum.
\newblock {\em J. Differential Equations}, 260(5):4440--4453, 2016.

\bibitem{MR3355116}
C\'esar~J. Niche and Mar\'\i a~E. Schonbek.
\newblock Decay characterization of solutions to dissipative equations.
\newblock {\em J. Lond. Math. Soc. (2)}, 91(2):573--595, 2015.

\bibitem{MR88624}
Franz Rellich.
\newblock Halbbeschr\"ankte {D}ifferentialoperatoren h\"oherer {O}rdnung.
\newblock In {\em Proceedings of the {I}nternational {C}ongress of
  {M}athematicians, 1954, {A}msterdam, vol. {III}}, pages 243--250. Erven P.
  Noordhoff N. V., Groningen, 1956.

\bibitem{MR240668}
Franz Rellich.
\newblock {\em Perturbation theory of eigenvalue problems}.
\newblock Gordon and Breach Science Publishers, New York-London-Paris, 1969.
\newblock Assisted by J. Berkowitz, With a preface by Jacob T. Schwartz.

\bibitem{MR571048}
Mar\'{\i}a~E. Schonbek.
\newblock Decay of solutions to parabolic conservation laws.
\newblock {\em Comm. Partial Differential Equations}, 5(7):449--473, 1980.

\bibitem{MR775190}
Mar\'{\i}a~E. Schonbek.
\newblock {$L^2$} decay for weak solutions of the {N}avier-{S}tokes equations.
\newblock {\em Arch. Rational Mech. Anal.}, 88(3):209--222, 1985.

\bibitem{MR837929}
Mar\'{\i}a~E. Schonbek.
\newblock Large time behaviour of solutions to the {N}avier-{S}tokes equations.
\newblock {\em Comm. Partial Differential Equations}, 11(7):733--763, 1986.

\end{thebibliography}
\end{document}